\documentclass{amsart}

\usepackage{amsmath,amssymb,cite}
\usepackage{amsthm}
\usepackage{amsrefs}
\usepackage{mathtools}
\usepackage{tikz}
\usetikzlibrary{patterns}
\usepackage{diagbox}
\usepackage{array}
\usepackage{graphicx}
\usepackage{relsize}
\usepackage{enumitem}
\usepackage[titletoc]{appendix}

\usepackage{caption}
\usepackage{subcaption}  
\usepackage{multirow}
\usepackage{tabularx}
\usepackage{tikz-cd}
\usepackage[plainpages=false,pdfborder={0 0 0}]{hyperref}

\newtheorem{thm}{Theorem}[section]
\newtheorem{lem}[thm]{Lemma}
\newtheorem{prop}[thm]{Proposition}
\newtheorem{conj}[thm]{Conjecture}
\theoremstyle{definition}

\newtheorem{exam}[thm]{Example}

\newtheorem{cor}[thm]{Corollary}
\theoremstyle{remark}
\newtheorem{rmk}[thm]{Remark}
\newcommand{\p}{\partial}

\numberwithin{equation}{section}

\begin{document}

\title{Static Manifolds with Boundary and Rigidity of Scalar Curvature and Mean Curvature}


\author{Hongyi Sheng}
\address{Department of Mathematics, University of California, San Diego, La Jolla, CA 92093, USA}
\email{hosheng@ucsd.edu}
\thanks{}


\date{}

\dedicatory{}

\begin{abstract}
On a compact manifold with boundary, the map consisting of the scalar curvature in the interior and the mean curvature on the boundary is a local surjection at generic metrics \cite{C-V}. Moreover, this result may be localized to compact subdomains in an arbitrary Riemannian manifold with boundary \cite{Sheng}. The non-generic case (also called non-generic domains) corresponds to static manifolds with boundary. We discuss their geometric properties, which also work as the necessary conditions of non-generic metrics. In space forms and the Schwarzschild manifold, we classify simple non-generic domains (with only one boundary component) and show their connection with rigidity theorems and the Schwarzschild photon sphere \cite{Cederbaum, Cederbaum-G, Claudel-V-E}.
\end{abstract}

\maketitle

\section{Introduction}\label{sec1}
Let $(M,g)$ be a complete, connected, smooth Riemannian manifold. A static potential is a non-trivial solution $u\in C^\infty(M)$ to the equation
\begin{equation}\label{static eqn}
    -\left(\Delta u\right) g+\operatorname{Hess}u - u \operatorname{Ric} = 0.
\end{equation}
Riemannian manifolds admitting static potentials are called static manifolds. In fact, static potentials are related to static spacetimes in general relativity. A static spacetime is a four-dimensional Lorentzian manifold which possesses a timelike Killing field and a spacelike hypersurface which is orthogonal to the integral curves of this Killing field. Moreover, we have the following proposition:
\begin{prop}[Corvino \cite{C}]
    Let $(M,g)$ be a Riemannian manifold. Then $u$ is a static potential if and only if the warped product metric $\bar{g}\equiv-u^2 \, dt^2+g$ is Einstein (away from the zeros of $u$).
\end{prop}

Static manifolds are interesting geometric objects that arise in other contexts as well. When Fischer and Marsden \cite{F-M} studied the local surjectivity of the scalar curvature operator, they also derived the left hand side of equation (\ref{static eqn}) as the formal $L^2$-adjoint of the linearization of the scalar curvature. Later Corvino \cite{C} considered localized deformations of metrics and extended the local surjectivity result at non-static metrics.

Static manifolds exhibit many interesting geometric properties. For example, the scalar curvature must be constant and the zero-set of a static potential is a totally geodesic hypersurface \cite{C}. Therefore, many people are interested in the classification of static manifolds. In particular, Fischer and Marsden \cite{F-M} raised the possibility of identifying all compact static manifolds. Now it is known
that in dimension $3$, besides flat tori $T^3$ and round spheres $S^3$, $S^1 \times S^2$ is also a compact static manifold. Later, other warped metrics on $S^1 \times_r S^2$ were found to be static \cite{Ejiri}. Kobayashi \cite{K} and Lafontaine \cite{Lafontaine} proved a more general classification result for locally conformally flat static manifolds. Qing and Yuan \cite{Q-Y} gave a classification result in dimension $3$ assuming a weaker hypothesis on the Cotton tensor. On the other hand, Ambrozio \cite{Ambrozio} computed a Bochner type formula and deduced classification results for compact static three-manifolds with positive scalar curvature.

It is natural to extend the above discussion to manifolds with boundary. Of course there are different ways to define so-called ``static manifolds with boundary" with different boundary conditions (see e.g. \cite[Equation (2.4)]{C-E-M}). In \cite{Sheng}, the author followed the ideas of Escobar \cite{E1, E2} and Corvino \cite{C} to consider localized deformations prescribing scalar curvature and boundary mean curvature. It turns out that the map is a
local surjection at generic metrics. So we are interested to know the properties of non-generic metrics when the map is not a local surjection. To be more precise, let $(\Omega^n, \Sigma^{n-1})$ be a complete, connected, smooth domain in $(M^n,g)$ where $\Sigma^{n-1}=\p\Omega^n$ and $n\ge2$. Then $(\Omega, \Sigma)$ is called a non-generic domain (or a ``static manifold with boundary" in our setting) if there is a non-trivial solution $u\in C^\infty(\overline{\Omega})$ to the boundary value problem:
\begin{align}\label{non-generic eqn}
\left\{\begin{aligned} 
-\left(\Delta u\right) g+\operatorname{Hess}u - u \operatorname{Ric} & = 0 \qquad & \text{in } \Omega\\ 
u_\nu \hat{g} & = u h & \text{on } \Sigma.
\end{aligned}\right.
\end{align}
Here $\hat{g}$ is the induced metric on $\Sigma$, $\nu$ is the outward unit normal to $\Sigma$, and $h(X,Y) = -\left<\nu, D_{X}Y\right>$ is the second fundamental form, where $X,Y$ are vector fields tangent to $\Sigma$. And we use $H = \operatorname{tr}_{\hat{g}} h$ to denote the mean curvature. Note that the interior equation is the same as (\ref{static eqn}), while the boundary equation is closely related to the linearization of the mean curvature operator (see \cite{Sheng} for more details). Henceforth, we may interchangeably use ``static manifolds with boundary" and ``non-generic domains" to denote the same concept.

With the above boundary condition, static manifolds with boundary exhibit richer geometric structures than the closed case. Sections \ref{sec2} and \ref{sec3} explore global properties, focusing primarily on compact cases.

Non-compact static manifolds with boundary, on the other hand, are often related to rigidity results in the positive mass theorem \cite{A-B-dL, A-dL, C-N, H-W, H-M-R, M, Q-Y2, S-T} and hold particular significance in general relativity. Therefore, we will try to classify non-generic domains in space forms and in the Schwarzschild manifold, and show their connection with rigidity theorems and the Schwarzschild photon sphere \cite{Cederbaum, Cederbaum-G, Claudel-V-E}.

The paper is organized as follows. In Section \ref{sec2}, we define basic notation and review previous results about non-generic domains. In Section \ref{sec3}, we focus on compact non-generic domains and discuss global properties. In Section \ref{sec4}, we use conformal methods to classify non-generic \textbf{simple} (having only one boundary component) domains in space forms and in the Schwarzschild manifold. We then discuss how to construct general non-generic domains with multiple boundary components.

\section{Preliminaries}\label{sec2}
In this section, we will review some basic properties of non-generic domains. 

Let $(\Omega^n, \Sigma^{n-1})$ be a complete, connected, smooth domain in $(M^n,g)$ where $\Sigma^{n-1}=\p\Omega^n$ and $n\ge2$, so that $(\Omega, \Sigma)$ itself is also a manifold with boundary. Then $(\Omega, \Sigma)$ is called a non-generic domain if there is a non-trivial solution $u\in C^\infty(\overline{\Omega})$ to the boundary value problem (\ref{non-generic eqn}). Many authors call a non-trivial solution to (\ref{static eqn}) a static potential. With the boundary condition, however, we will call a non-trivial solution to (\ref{static eqn}) a \textbf{possible} static potential on $(\Omega, \Sigma)$, and if it satisfies the boundary condition on $\Sigma$ as well, we then call it a static potential on $(\Omega, \Sigma)$ (see \cite{A-dL, A-dL2, C-N} for similar definitions of static potentials and static manifolds with boundary). That is, a non-trivial solution to (\ref{non-generic eqn}) is called a static potential in our setting, and a domain $(\Omega, \Sigma)$ admitting a static potential is called non-generic.

We define as in \cite{Sheng} the operator $L^*: C^{\infty}(\overline{\Omega}) \longrightarrow \mathcal{C}^{\infty}(\overline{\Omega})$ by
$$
L^*u = -\left(\Delta u\right) g+\operatorname{Hess}u - u \operatorname{Ric},
$$
and the operator $\Phi^*: C^{\infty}(\overline{\Omega}) \longrightarrow \mathcal{C}^{\infty}(\overline{\Omega}) \times \mathcal{C}^{\infty}(\Sigma)$ by
$$
\Phi^*u = (L^*u, u_\nu \hat{g}-u h).
$$
Here $\mathcal{C}^{\infty}(\overline{\Omega})$ denotes the space of symmetric $(0,2)$-tensor fields on $\overline{\Omega}$ that are smooth. With this notation, the space of static potentials on $(\Omega, \Sigma)$ can now be denoted as $\operatorname{ker} \Phi^*$, and $(\Omega, \Sigma)$ is a non-generic domain if and only if $\operatorname{ker} \Phi^*\ne0$.

Let us first recall some geometric properties of non-generic domains derived in \cite{Sheng}.

\begin{thm}\label{const mean}
If $(\Omega, \Sigma)$ is a non-generic domain, then the scalar curvature of $\Omega$ is constant, the boundary $\Sigma$ is umbilic, and the mean curvature is constant on $\Sigma$.
\end{thm}

\begin{rmk}
    The same proof in \cite[Theorem 3.3]{Sheng} actually gives $\operatorname{Ric}_{i\nu} = 0 \,\,\, (i = 1, \ldots, n-1)$ on $\Sigma$ as well.
\end{rmk}

The next two results use the fact that the equation $L^*u = 0$ reduces to a second-order linear ODE along geodesics, following the proof of \cite[Corollary 2.4]{C}, cf. \cite{F-M}. The map which takes a possible static potential $u$ to its one-jet $(u(p), du(p))$ at a point is injective, so that the dimension of the space of possible static potentials is at most $n+1$. The boundary condition here shows the map $u \mapsto \left(u(p), \left.du\right|_\Sigma(p)\right)$ is injective on static potentials.

\begin{prop}
For non-generic domains $(\Omega, \Sigma)$ in $(M^n,g)$, we have $\operatorname{dim} \operatorname{ker} \Phi^{*} \le n$ in $H_{\mathrm{loc}}^{2}(\Omega\cup\Sigma)$.
\end{prop}
As we will see in Section \ref{sec4}, the upper bound is achieved by simple non-generic domains in space forms.

\begin{prop}
$H_{\mathrm{loc}}^{2}$ elements in $\operatorname{ker} \Phi^*$ are actually in $C^2(\overline{\Omega})$.
\end{prop}

In \cite[Lemma 2.12]{Sheng} we also derived a very useful lemma.

\begin{lem}\label{no kernel}
There is no non-trivial solution in $H_{\operatorname{loc}}^{2}(\Omega\cup\Sigma)$ to the following equations:
$$
\left\{\begin{aligned} 
L^*u & = 0 \qquad\,\text{in } \Omega\\ 
u = u_\nu & = 0 \qquad\text{on } \Sigma.
\end{aligned}\right.
$$
\end{lem}
This follows readily, since the interior equation $L^*u = 0$ reduces to an ODE along geodesics. A solution $u$ extends to the boundary by the preceding regularity result, and the boundary conditions give trivial initial conditions for the ODE along any geodesic. And as a quick corollary, for non-generic domains with totally geodesic boundary, the space of static potentials is fully determined by the boundary data on $\Sigma$.
\begin{cor}
Let $(\Omega, \Sigma)$ be a non-generic domain. If $H=0$, then $\Sigma$ is totally geodesic, and an element of $\operatorname{ker} \Phi^*$ is fully determined by its boundary data on $\Sigma$.
\end{cor}
\begin{proof}
If $H=0$, then for any $u\in\operatorname{ker} \Phi^*$, we have
$$
L^*u = 0 \quad \text {in }  \Omega\qquad\text{and}\qquad u_\nu = 0  \quad\text {on }  \Sigma.
$$
So for any $u_1, u_2\in\operatorname{ker} \Phi^*$ with $u_1\equiv u_2$ on $\Sigma$, consider $v = u_1-u_2\in\operatorname{ker} \Phi^*$. We have 
    $$
\left\{\begin{aligned} 
L^*v & = 0 \qquad\,\text{in } \Omega\\ 
v = v_\nu & = 0 \qquad\text{on } \Sigma.
\end{aligned}\right.
$$
By Lemma \ref{no kernel}, $v\equiv0$ in $\Omega$. This means $u_1\equiv u_2$ in $\Omega$.
\end{proof}

\section{Compact Non-generic Domains}\label{sec3}
Let us now focus on compact non-generic domains. 

Taking the trace of (\ref{non-generic eqn}), we get
\begin{align}\label{generic trace}
\left\{\begin{aligned}
\Delta u+\frac{R}{n-1} u & =0 \quad & \text {in }  \Omega\\
u_\nu-\frac{H}{n-1} u & =0 & \text {on }  \Sigma.
\end{aligned}\right.
\end{align}
This means, for any static potential $u$, they must first satisfy (\ref{generic trace}). It is very useful when we want to get some global properties of compact non-generic domains.

From Theorem \ref{const mean} we know the scalar curvature $R$ of $\Omega$ and the mean curvature $H$ of $\Sigma$ are constant on non-generic domains. Thus all the non-generic domains may be divided into 9 cases, according to whether $R, H$ are positive, negative or zero. However, some of the cases are not possible. Cruz and Vit\'{o}rio \cite{C-V} found that:
\begin{prop}[Cruz-Vit\'{o}rio]
Let $(\Omega, \Sigma)$ be a compact non-generic domain.
\begin{itemize}
\item[(i)] If $R = 0$, then $H \ge 0$; and if $H = 0$ as well, then $h \equiv 0$.
\item[(ii)] If $H = 0$, then $R \ge 0$; and if $R = 0$ as well, then $\operatorname{Ric} \equiv 0$.
\end{itemize}
\end{prop}

Note that:
\begin{itemize}
\item[(i)] If $R = 0$ and $H > 0$, then $\frac{H}{n-1}$ is a Steklov eigenvalue, which basically means
$$
\Delta u =0 \quad \text {in }  \Omega\qquad\text{and}\qquad u_\nu =\frac{H}{n-1} u  \quad\text {on }  \Sigma.
$$
In this case, $\operatorname{ker}\Phi^*$ lies in the eigenspace corresponding to the Steklov eigenvalue $\frac{H}{n-1}$.
\item[(ii)] If $R > 0$ and $H = 0$, then $\frac{R}{n-1}$ is an eigenvalue of the Neumann boundary value problem; in this case, $\operatorname{ker}\Phi^*$ lies in the eigenspace corresponding to the eigenvalue $\frac{R}{n-1}$.
\item[(iii)*] If $R = H = 0$, then $g$ is Ricci flat, $\Sigma$ is totally geodesic, and $\operatorname{ker}\Phi^*$ consists of constant functions on $\overline{\Omega}$. In this case, the boundary scalar curvature $R_\Sigma:= R_{\hat{g}}$ also vanishes.
\end{itemize}
\begin{prop}
    Let $(\Omega, \Sigma)$ be a compact non-generic domain. If $R = H = 0$, then the boundary scalar curvature $R_\Sigma = 0$.
\end{prop}
\begin{proof}
By looking at the $\nu\nu$-th term of $L^*u=0$, we have 
$$
-\Delta u + u_{\nu\nu} - uR_{\nu\nu} = 0.
$$
Since $u$ is a non-zero constant, $R_{\nu\nu}$ must be 0. Then by the Gauss equation, 
$$
R_\Sigma = R - 2R_{\nu\nu} - \|h\|^2 + H^2 = 0.
$$    
\end{proof}

We may further rule out some more cases:
\begin{prop}
Let $(\Omega, \Sigma)$ be a compact non-generic domain. If $R<0$, then $H>0$; if $H<0$, then $R>0$.
\end{prop}
\begin{proof}
If $0\ne u\in\operatorname{ker}\Phi^*$, then $u$ must satisfy (\ref{generic trace}). Multiplying $u$ to the first equation and integrating it over $\Omega$, we get, with $Du = \operatorname{grad}_g u$ the gradient of $u$,
$$
\begin{aligned}
0 &=\int_{\Omega}\left(u \Delta u+\frac{R}{n-1} u^{2}\right) d \mu_g\\
	&=\int_{\Omega}\left(-\left|D u\right|^{2}+\frac{R}{n-1} u^{2}\right) d \mu_g+\int_{\Sigma} u u_\nu\,d\sigma_g\\
	&=\int_{\Omega}\left(-\left|D u\right|^{2}+\frac{R}{n-1} u^{2}\right) d \mu_g+\int_{\Sigma} \frac{H}{n-1} u^{2}\,d\sigma_g.
\end{aligned}
$$
Hence, if $R<0$, the first integral is strictly negative, then $H$ must be strictly positive. Similarly, if $H<0$, then $R>0$.
\end{proof}

As a summary, we have the following diagram for compact non-generic domains:
\begin{table}[!htbp]
\centering
\begin{tabular}{|c|c|c|c|}
\hline
\diagbox{H}{R}&$+$&$0$&$-$\\ 
\hline
$+$&?&Steklov&?\\
\hline
$0$&Neumann&*&$\times$\\
\hline
$-$&?&$\times$&$\times$\\
\hline
\end{tabular}
 \caption{Cases of compact non-generic domains}
  \label{tab: compact non-generic domains}
\end{table}



In the next section, we will see some more examples that will fill up Table \ref{tab: compact non-generic domains} above.

\section{Non-generic Domains in Space Forms and in the Schwarzschild Manifold}\label{sec4}
An umbilic hypersurface in a space form has constant mean curvature \cite[p.~182 Exercise 6a]{dC}, thus a region bounded by an umbilic hypersurface in a space form satisfies all the necessary conditions of a non-generic domain. We will see that such regions are indeed non-generic domains.

For simplicity, we will first consider non-generic domains with only one boundary component. We call them \textbf{simple} non-generic domains. If a non-generic domain has multiple boundary components, then it can be viewed as the intersection of simple non-generic domains, as long as their boundaries do not intersect with each other. We will briefly discuss ``non-simple" non-generic domains and look at an example in Section \ref{non-simple}.

\subsection{The Euclidean space $\mathbb{R}^n$}
Consider the flat metric on the Euclidean space $\mathbb{R}^n$. First of all, notice that we have the following result:
\begin{prop}
The possible static potentials on non-generic domains in $\mathbb{R}^{n}$ are precisely those functions in $\operatorname{span}\{1, x_1, \cdots, x_n\}$, where $(x_1, \cdots, x_{n})$ are the coordinates of $\mathbb{R}^{n}$.
\end{prop}
\begin{proof}
For $i = 1, \cdots, n$, it is clear that $L^*(x_i) = 0 = L^*(1)$. On the other hand, they are the only possible static potentials, since for any static potential $u$,
$$
L^*u = -\left(\Delta u\right) g + \operatorname{Hess}u = 0.
$$
From (\ref{generic trace}) we know $\Delta u = 0$. So this means $\operatorname{Hess}u = 0$, which implies $u$ is a linear combination of the coordinates $x_i$ and constant functions.
\end{proof}

As for totally umbilical hypersurfaces $\Sigma\subset\mathbb{R}^n$, they can only be spheres or planes \cite[p.~183 Exercise 6c]{dC}. Cruz and Vit\'{o}rio \cite{C-V} found that the Euclidean ball with its boundary is a non-generic domain in $\mathbb{R}^n$, and later Ho and Huang \cite{H-H} determined the space of static potentials. We will do likewise for the exterior of a Euclidean ball, and the upper half-space. 

\begin{prop}
Let $S_r$ be an $(n-1)$-dimensional Euclidean sphere of radius $r$ in $\mathbb{R}^{n}$. Then both the interior and the exterior region of $S_r$ are non-generic domains in $\mathbb{R}^{n}$. Moreover,
$$\operatorname{ker}\Phi^* = \operatorname{span}\{x_1, \cdots, x_n\}.$$

Let $(\Omega, \Sigma)$ be the upper half-space, then it is also a non-generic domain in $\mathbb{R}^{n}$. Moreover, if we parametrize $\Omega=\{x_n > 0\}$ and $\Sigma = \{x_n = 0\}$, then 
$$\operatorname{ker}\Phi^* = \operatorname{span}\{1, x_1, \cdots, x_{n-1}\}.$$
\end{prop}
\begin{proof}
Consider the exterior region of $S_r$, then on $S_r$ we have for $i=1, \cdots, n$,
$$
(x_i)_\nu-\frac{H}{n-1} x_i = Dx_i \cdot \left(-\frac{x}{r} \right) + \frac1r x_i = -\frac{x_i}{r} + \frac{x_i}{r} = 0, 
$$
while
$$
(1)_\nu-\frac{H}{n-1} \cdot 1 = 0 + \frac1r  = \frac1r > 0.
$$
Thus, the exterior region of $S_r$ is a non-generic domain in $\mathbb{R}^n$ with $\operatorname{ker}\Phi^* = \operatorname{span}\{x_1, \cdots, x_n\}$, and similarly for the interior region of $S_r$.

Let $(\Omega, \Sigma)$ be the upper half-space, then for $1\le i\le n-1$,
$$
(x_i)_\nu-\frac{H}{n-1} x_i = \frac{\p x_i}{\p x_n} + 0 = 0,
$$
$$
\text{and}\quad (1)_\nu-\frac{H}{n-1} \cdot 1 = 0 + 0 = 0,
$$
while
$$
(x_n)_\nu-\frac{H}{n-1} x_n = 1 + 0 =1 > 0.
$$
Thus, the upper half-space is a non-generic domain in $\mathbb{R}^n$, and $\operatorname{ker}\Phi^* = \operatorname{span}\{1, x_1, \cdots, x_{n-1}\}.$
\end{proof}

As a summary, we have the following classification result:
\begin{thm}[Classification of simple non-generic domains in $\mathbb{R}^n$]
There are two types of boundaries of non-generic domains in $\mathbb{R}^n$: (i) spheres,  (ii) hyperplanes.
\end{thm}

In 2016, Almaraz-Barbosa-de Lima \cite{A-B-dL} had the following conjecture of a positive mass theorem for asymptotically flat manifolds with a non-compact boundary, and they were able to prove the case when either $3 \le n \le 7$ or $n \ge 3$ and $M$ is spin.
\begin{conj}[Almaraz-Barbosa-de Lima]
If $(M, g)$ is asymptotically flat and satisfies $R \ge 0$ and $H \ge 0$, then $m(M,g) \ge 0$, with the equality occurring if and only if $(M, g)$ is isometric to $(\mathbb{R}^n_+, \delta)$.
\end{conj}

They also had the following proposition regarding the rigidity part:
\begin{prop}[Almaraz-Barbosa-de Lima]
Let $(M,g)$ be asymptotically flat with $3 \le n \le 7$ or $n \ge 3$ and $M$ spin, and satisfy $R \ge 0$ and $H \ge 0$. Assume further that there exists a compact subset $K \subset M$ such that $(M \setminus K, g)$ is isometric to $(\mathbb{R}^n_+\setminus \overline{B_1^+}(0), \delta)$. Then $(M, g)$ is isometric to $(\mathbb{R}^n_+, \delta)$.
\end{prop}

Thanks to the above proposition, we can now better understand why the half-space is a non-generic domain in $\mathbb{R}^n$. Suppose not, then the half-space with $R(g)=0$ and $H(g)=0$ satisfies generic conditions. This means we would be able to perform a localized deformation \cite[Theorem 2.9]{Sheng} on a compact region $K\subset\mathbb{R}^n_+$. More precisely, for any $S\ge0$ and $H'\ge0$ small enough such that $S$ is supported in $K$ and $H'$ is supported in $K\cap\{x_n=0\}$, there is a metric $g'$ with $R(g') = S$ in $K$,  $H(g') = H'$ in $K\cap\{x_n=0\}$, and $g' \equiv g$ outside $K$. However, this is a contradiction to the rigidity theorem.

There is another generalized version of the positive mass theorem \cite{M, S-T} that will imply the following rigidity theorem for the unit ball in $\mathbb{R}^n$:
\begin{prop}[Hang-Wang \cite{H-W}]
Let $(M,g)$ be an $n$-dimensional compact Riemannian manifold with boundary and the scalar curvature $R \ge 0$. The boundary is isometric to the standard sphere $\mathbb{S}^{n-1}$ and has mean curvature $n - 1$. Then $(M,g)$ is isometric to the unit ball in $\mathbb{R}^n$. (If $n > 7$, we also assume $M$ is spin.)
\end{prop}

Notice that it can also give a good explanation of why the unit ball is a non-generic domain in $\mathbb{R}^n$.

\subsection{The unit sphere $\mathbb{S}^n$}
Consider the standard metric $g_{\mathbb{S}^n}$ on the unit sphere $\mathbb{S}^n$. In this case, the sectional curvature $K = 1$, the Ricci curvature $\operatorname{Ric} = (n-1)g_{\mathbb{S}^n}$, and the scalar curvature $R = n(n-1)$. 

First of all, notice that we have the following result:
\begin{prop}
The linear combinations of the coordinates $x_1, \cdots, x_{n+1}$ are the only possible static potentials on non-generic domains in $\mathbb{S}^n$.
\end{prop}
\begin{proof}
For $i=1, \cdots, n+1$,
$$
\operatorname{Hess}_{g_{\mathbb{S}^n}} x_{i}+x_{i} g_{\mathbb{S}^n}=0.
$$
Thus,
$$
\begin{aligned}
L^*(x_i) 	& = -\left(\Delta_{g_{\mathbb{S}^n}} x_i\right) g_{\mathbb{S}^n}+\operatorname{Hess}_{g_{\mathbb{S}^n}}x_i - (n-1) x_i {g_{\mathbb{S}^n}}\\
	& = -\left(\Delta_{g_{\mathbb{S}^n}} x_i\right) g_{\mathbb{S}^n} - x_{i} g_{\mathbb{S}^n} - (n-1) x_i {g_{\mathbb{S}^n}}\\
	& = -\left(\Delta_{g_{\mathbb{S}^n}} x_i + nx_i\right) g_{\mathbb{S}^n}\\
	& = 0.
\end{aligned}
$$
On the other hand, the space of possible static potentials has dimension at most $n + 1$ \cite[Corollary 2.4]{C}, so the result follows immediately.
\end{proof}

As for totally umbilical hypersurfaces $\Sigma\subset\mathbb{S}^n$, they can only be spherical caps. This follows from the classification in $\mathbb{R}^n$ (spheres or planes) together with stereographic projection. Cruz and Vit\'{o}rio \cite{C-V} found that the upper hemisphere with its boundary is a non-generic domain in $\mathbb{S}^n$, and later Ho and Huang \cite{H-H} determined the space of static potentials. We will do likewise for all spherical caps.

\begin{thm}\label{sphere}
Let $(\Omega, \Sigma)$ be an $n$-dimensional spherical cap equipped with the standard metric $g_{\mathbb{S}^n}$. Then $(\Omega, \Sigma)$ is a non-generic domain in $\mathbb{S}^n$. Moreover, if the north pole $N$ is the center of the spherical cap, then
$$\operatorname{ker}\Phi^* = \operatorname{span}\{x_1, \cdots, x_n\}.$$
In fact, spherical caps are the only simple non-generic domains in $\mathbb{S}^n$.
\end{thm}

\begin{proof}
If the north pole $N$ is the center of the spherical cap, then $\Sigma$ may be parametrized as $\{\mathbf{x}\in\mathbb{S}^n: x_{n+1} = \operatorname{cos}\theta\}$, so that any point $\mathbf{x}\in\Sigma$ can be written as $\mathbf{x} = (\operatorname{sin}\theta\,\mathbf{\xi}, \operatorname{cos}\theta)$, where $\mathbf{\xi}\in\mathbb{S}^{n-1}$, and the induced metric on $\Sigma$ is $\hat{g} = \operatorname{sin}^2\theta \,g_{\mathbb{S}^{n-1}}$. In this case, the outward unit normal $\nu = \frac{\p}{\p \theta}$, and the mean curvature is
\begin{align*}
H & = \frac{\p}{\p \theta} \operatorname{ln}\sqrt{\operatorname{det}\hat{g}}\\
	& = \frac{\p}{\p \theta} \operatorname{ln}\sqrt{\left(\operatorname{sin}^2\theta\right)^{n-1}\operatorname{det}g_{\mathbb{S}^{n-1}}}\\
	& = \frac{\p}{\p \theta}\left( (n-1)\operatorname{ln}(\operatorname{sin}\theta)+ \operatorname{ln}\sqrt{\operatorname{det}g_{\mathbb{S}^{n-1}}}\right)\\
	& = (n-1)\frac{\operatorname{cos}\theta}{\operatorname{sin}\theta}.
\end{align*}
This means, for $i = 1, \cdots, n$,
$$
(x_i)_\nu - \frac{H}{n-1} x_i = \operatorname{cos}\theta\,\xi_i - \frac{\operatorname{cos}\theta}{\operatorname{sin}\theta} \operatorname{sin}\theta\,\xi_i = 0,
$$
while
$$
(x_{n+1})_\nu - \frac{H}{n-1} x_{n+1} = \frac{\p}{\p \theta}\operatorname{cos}\theta - \frac{\operatorname{cos}\theta}{\operatorname{sin}\theta}\operatorname{cos}\theta = - \frac1{\operatorname{sin}\theta}<0.
$$
Thus, $\operatorname{ker}\Phi^* = \operatorname{span}\{x_1, \cdots, x_n\}.$
\end{proof}


For spherical caps, the scalar curvature $R>0$; when $\theta$ varies between $0$ and $\pi$, the range of the mean curvature $H$ is the set of all real numbers $\mathbb{R}$.

We would like to note that non-generic domains in $\mathbb{S}^n$ are related to the famous Min-Oo’s Conjecture \cite{M-O}:
\begin{conj}[Min-Oo]
Let $(M,g)$ be an $n$-dimensional compact Riemannian manifold with boundary and the scalar curvature $R \ge n(n - 1)$. The boundary is isometric to the standard sphere $\mathbb{S}^{n-1}$ and is totally geodesic. Then $(M,g)$ is isometric to the hemisphere $\mathbb{S}^n_+$.
\end{conj}

Min-Oo’s Conjecture has been verified in many special cases (see e.g. \cite{Brendle-M, E, H-W, H-W2, M-T}), however, counterexamples were also constructed \cite{Brendle-M-N, C-E-M}.

\subsection{The hyperbolic space $\mathbb{H}^n$}
Consider the hyperboloid model for the hyperbolic space $\mathbb{H}^n$. To be more precise, let us consider the Minkowski quadratic form defined on $\mathbb{R}^{1, n}$ by
$$
Q\left(x_{0}, x_{1}, \ldots, x_{n}\right)=-x_{0}^{2}+x_{1}^{2}+\cdots+x_{n}^{2}.
$$
Then the hyperbolic space is defined by
$$\mathbb{H}^n = \{\mathbf{x}\in\mathbb{R}^{1, n}: Q(\mathbf{x}) = -1, x_0>0\}$$
with induced metric $g=g_{\mathbb{H}^n}$. In this case, the sectional curvature $K = -1$, the Ricci curvature $\operatorname{Ric} = -(n-1)g$, and the scalar curvature $R = -n(n-1)$. 

As a part of the unit sphere in the Minkowski space, the hyperbolic space has many properties that are similar to the unit sphere $\mathbb{S}^n$ in the Euclidean space. 
\begin{prop}\label{hyperbolic static potential}
The possible static potentials on a non-generic domain in $\mathbb{H}^n$ lie in the span of the restrictions of the coordinates $x_0, x_1, \cdots, x_n$ of $\mathbb{R}^{1, n}$ to $\mathbb{H}^n$.
\end{prop}
\begin{proof}
Let us consider the hyperbolic space as a graph in $\mathbb{R}^{1, n}$. To be more precise, any point $(x_0, \mathbf{x}) = (x_0, x_{1}, \ldots, x_{n}) \in \mathbb{H}^n$ satisfies
$$
x_0 = \sqrt{1+\|\mathbf{x}\|^2},
$$
thus the hyperbolic space can be characterized as the image of the map $F: \mathbb{R}^{n}\to\mathbb{R}^{1, n}$ given by
$$
F(\mathbf{x})=(\sqrt{1+\|\mathbf{x}\|^2}, \mathbf{x}).
$$

Denote $e_i\in T\mathbb{H}^n \, (i = 1,\cdots, n)$ as 
$$
e_i = dF\left(\frac{\p}{\p x_i}\right) = \frac{x_i}{x_0}\frac{\p}{\p x_0} + \frac{\p}{\p x_i}.
$$
Then the induced metric is
$$
g_{ij} = \left<e_i, e_j\right> = \delta_{ij} - \frac{x_ix_j}{x_0^2}
\qquad
\text{and}
\qquad
g^{ij} = \delta^{ij} + x_ix_j.
$$
For $k = 1,\cdots, n$,
$$
\begin{aligned}
\operatorname{Hess}_{ij}x_k & = e_i(e_j x_k) - \Gamma_{ij}^l e_lx_k\\
	& = e_i(\delta_{jk}) - \Gamma_{ij}^l\delta_{lk}\\
	& = - \Gamma_{ij}^k\\
	& = -\frac{1}{2} g^{kl}\left(\frac{\partial}{\partial x_{j}} g_{l i}+\frac{\partial}{\partial x_{i}} g_{l j}-\frac{\partial}{\partial x_{l}} g_{i j}\right),
\end{aligned}
$$
and 
$$
\frac{\partial}{\partial x_{j}} g_{l i} = \frac{\partial}{\partial x_{j}}\left(\delta_{li} - \frac{x_lx_i}{x_0^2}\right) = -\frac{\delta_{jl}x_i + \delta_{ji}x_l}{x_0^2} + 2\frac{x_lx_ix_j}{x_0^4}.
$$
So
$$
\begin{aligned}
\operatorname{Hess}_{ij}x_k & = \frac{1}{2} g^{kl}\left(\frac{\delta_{jl}x_i + \delta_{ji}x_l}{x_0^2} - 2\frac{x_lx_ix_j}{x_0^4} +\frac{\delta_{il}x_j + \delta_{ji}x_l}{x_0^2} - 2\frac{x_lx_ix_j}{x_0^4}\right)  + \frac{1}{2} g^{kl}\left(-\frac{\delta_{jl}x_i + \delta_{li}x_j}{x_0^2} + 2\frac{x_lx_ix_j}{x_0^4}\right)\\
	& = \sum_l(\delta^{kl} + x_kx_l)\left(\frac{\delta_{ji}x_l}{x_0^2} - \frac{x_lx_ix_j}{x_0^4}\right)\\
	& = \frac{\delta_{ji}x_k}{x_0^2} + \frac{\sum_lx_k\delta_{ji}x_l^2}{x_0^2} - \frac{x_kx_ix_j}{x_0^4} - \frac{\sum_lx_kx_ix_jx^2_l}{x_0^4}\\
	& = x_k \delta_{ij} - \frac{x_kx_ix_j}{x_0^2}\\
	& = x_k g_{ij} = - \Gamma_{ij}^k,
\end{aligned}
$$
and
$$
\begin{aligned}
\operatorname{Hess}_{ij}x_0 & = e_i(e_j x_0) - \Gamma_{ij}^l e_lx_0\\
	& = e_i\left(\frac{x_j}{x_0}\right) - \Gamma_{ij}^l\frac{x_l}{x_0}\\
	& = \frac{\delta_{ij}x_0 - x_j\frac{x_i}{x_0}}{x_0^2} + \sum_l x_l g_{ij}\frac{x_l}{x_0}\\
	& = \frac1{x_0}g_{ij} + \frac1{x_0}g_{ij}\sum_l x^2_l\\
	& = x_0 g_{ij}.
\end{aligned}
$$
Thus for $i = 0, 1,\cdots, n$,
$$
\operatorname{Hess}_{g_{\mathbb{H}^n}}x_i = x_i g_{\mathbb{H}^n},
$$
and we have
\begin{equation*}
    L^*(x_i) = -\left(\Delta x_i\right) g + x_i g + (n-1)x_i g = \left(-\Delta x_i + nx_i\right) g = 0.
\end{equation*}
On the other hand, we know $\operatorname{dim} \operatorname{ker} L^{*} \le n+1$ in $H_{\mathrm{loc}}^{2}(\Omega)$ \cite[Corollary 2.4]{C}. This means the possible static potentials lie in the span of the restrictions of the coordinates $x_0, x_1, \cdots, x_n$ of $\mathbb{R}^{1, n}$ to $\mathbb{H}^n$.
\end{proof}


The totally umbilical hypersurfaces $\Sigma\subset\mathbb{H}^n$ are well-known \cite[pp.~177-185]{dC}, as we now recall. In the hyperboloid model, they are of the form $P\cap\mathbb{H}^n$, where $P$ is some affine hyperplane of $\mathbb{R}^{1, n}$. Alternatively, if we consider the upper half-space model, then $\Sigma$ must be a geodesic sphere, a horosphere, a hypersphere, or the intersection with $\mathbb{H}^{n}$ of hyperplanes of $\mathbb{R}^{n}$.

We proceed to check whether each of these umbilical hypersurfaces gives us a boundary of a non-generic domain in $\mathbb{H}^{n}$. Up to hyperbolic isometry, we can consider $\Sigma = S\cap\mathbb{H}^n$, where $S\subset\mathbb{R}^n$ is a Euclidean $(n-1)$-sphere of radius $1$. If the center of $S$ is in $\mathbb{H}^n$, the normalized mean curvature $\frac{H}{n-1}$ of such $\Sigma = S\cap\mathbb{H}^n$ is found to be as follows \cite[p.~184 Exercise 6e]{dC}:
\begin{itemize}
\item[(i)] $1$, if $S$ is tangent to $\p\mathbb{H}^n$ (horosphere);
\item[(ii)] $\operatorname{cos}\alpha$, if $S$ makes an angle $\alpha$ with $\p\mathbb{H}^n$ (hypersphere, see Fig. \ref{fig: hypersphere});
\item[(iii)] the Euclidean height $\eta>1$ of the center of $S$ relative to $\p\mathbb{H}^n$, if $S\subset\mathbb{H}^n$ (geodesic sphere).
\end{itemize}
If the center of $S$ is on $\p\mathbb{H}^n$, $\Sigma$ gives a totally geodesic hypersurface; if the center of $S$ is below $\p\mathbb{H}^n$ (i.e. the Euclidean height $-1<\eta< 0$), we have similar results by reflection around $\{y_n = 0\}$. So as a summary, we have
\begin{lem}\label{hyperbolic mean curv}
Consider the upper half-space model of $\mathbb{H}^n$. Let $\Sigma = S\cap\mathbb{H}^n$ be the intersection of $\mathbb{H}^n$ with a Euclidean $(n-1)$-sphere $S\subset\mathbb{R}^n$ of radius $1$. Then the normalized mean curvature $\frac{H}{n-1}$ of $\Sigma$ is the Euclidean height $\eta>-1$ of the center of $S$.
\end{lem}

\begin{figure}
\begin{center}
\includegraphics[scale = 0.075]{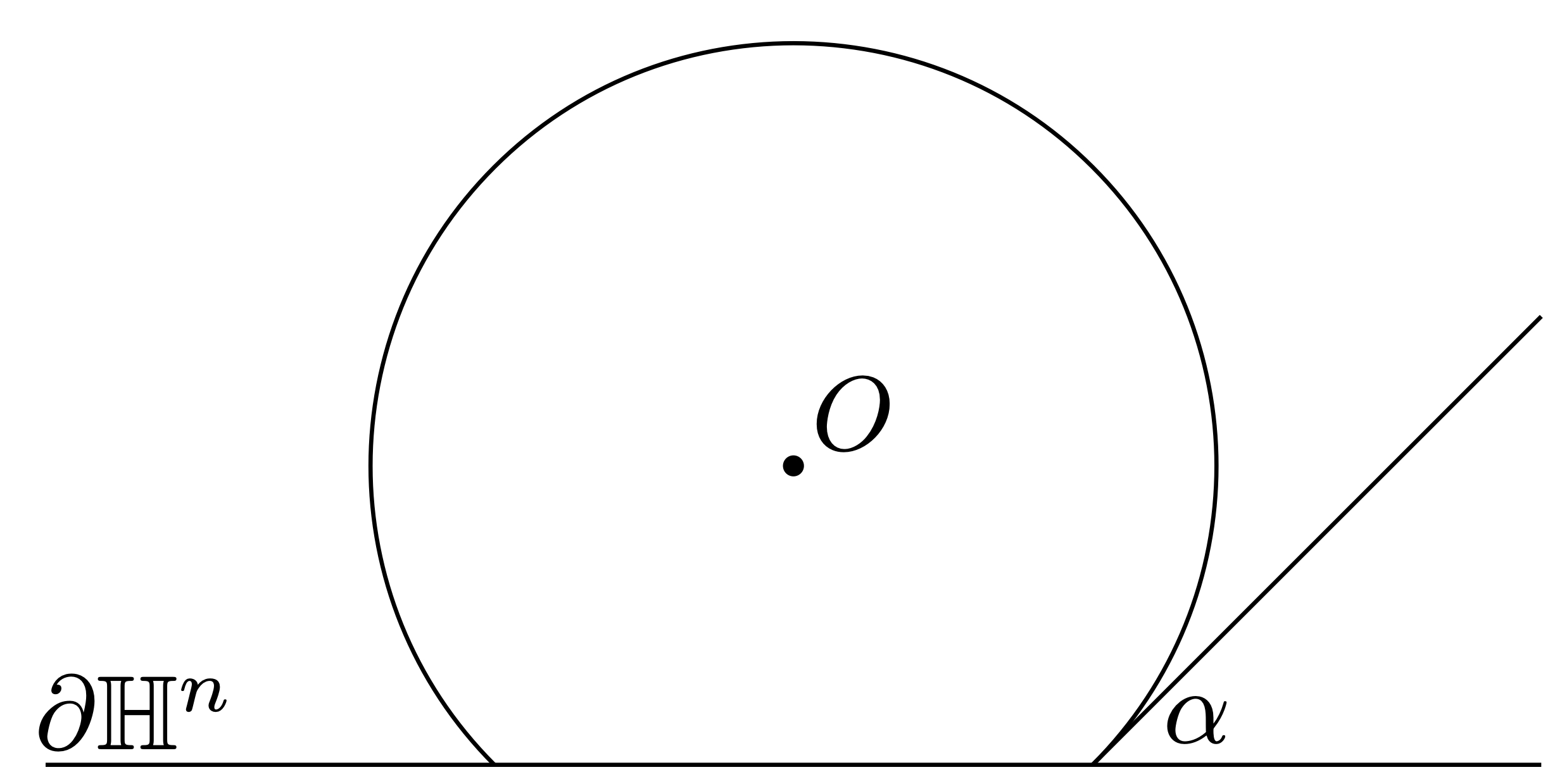}
\caption{A hypersphere in the half-space model, where $O$ is its Euclidean center}
\label{fig: hypersphere}
\end{center}
\end{figure}

Proposition \ref{hyperbolic static potential} tells us the the only possible static potentials are the coordinates $x_0, x_1, \cdots, x_n$ of $\mathbb{R}^{1, n}$ restricted to $\mathbb{H}^n$, so we first rewrite these coordinates in terms of the coordinates in the upper half-space model. Denote $\mathbb{H}^n_{-1}$ as the hyperboloid model, $D^n$ as the unit ball model, and $\mathbb{H}^n$ as the upper half-space model of the hyperbolic space. In \cite[p.~184 Exercise 7]{dC}, a stereographic projection from $\mathbb{H}^n_{-1}$ to $D^n$ was defined. Let us consider its inverse map $f_1: D^n\to\mathbb{H}^n_{-1}$ expressed $f_1(u_1,\cdots, u_n) = (x_0, x_1, \cdots, x_n)$ given by, where $i$ runs from $1$ to $n$:
$$
x_0 =\frac{2}{1-\sum_{i=1}^n u_i^{2}}-1,\qquad x_i =\frac{2 u_i}{1-\sum_{i=1}^n u_i^{2}}.
$$

A conformal mapping $f_2: \mathbb{H}^n\to D^n$ between the unit ball $D^n$ and the upper half-space $\mathbb{H}^n$, expressed $f_2(y_1, \cdots, y_n) = (u_1, \cdots, u_n)$, is given by, where $m$ runs from $1$ to $n-1$:
$$
u_m =\frac{2 y_m}{\sum_{m=1}^{n-1} y_m^{2} + (y_n + 1)^2},\qquad u_n =1 - \frac{2(y_n + 1)}{\sum_{m=1}^{n-1} y_m^{2} + (y_n + 1)^2}.
$$

Combining them together, we get the expression of the coordinates $x_0, x_1, \cdots, x_n$ in the upper half-space $\mathbb{H}^n$:
$$
\begin{aligned}
x_0 &=\frac{\sum_{m=1}^{n-1} y_m^{2} + y_n^2 + 1}{2y_n},\\
x_m &=\frac{y_m}{y_n},\\
x_n &=\frac{\sum_{m=1}^{n-1} y_m^{2} + y_n^2 - 1}{2y_n}.
\end{aligned}
$$

\begin{prop}
Consider the upper half-space model of $\mathbb{H}^n$. Let $\Sigma = S\cap\mathbb{H}^n$ be the intersection of $\mathbb{H}^n$ with a Euclidean $(n-1)$-sphere $S\subset\mathbb{R}^n$. Then both the interior region and the exterior region with boundary $\Sigma$ are non-generic domains in $\mathbb{H}^n$. Moreover, if the Euclidean radius of $S$ is $1$ and the Euclidean height of the center of $S$ is $\eta>-1$, then
$$\operatorname{ker}\Phi^* = \operatorname{span}\{x_1, \cdots, x_{n-1}, (2-\eta^2)x_0 + \eta^2\, x_n\}.$$
Here $(x_0, x_1, \cdots, x_n)$ are the coordinates of $\mathbb{R}^{1, n}$ restricted to $\mathbb{H}^n$ in the hyperboloid model.
\end{prop}
\begin{proof}
We will just verify that the interior region bounded by $\Sigma$ is a non-generic domain in $\mathbb{H}^n$; the same calculation will give us that the exterior region bounded by $\Sigma$ is a non-generic domain as well.

Let us use $y_i, i = 1,\cdots,n,$ for the coordinates of $\mathbb{R}^n$. The metric of the upper half-space model of $\mathbb{H}^n$ is $g = y_n^{-2}\delta$, where $\delta$ is the Euclidean metric. Up to isometry we may assume the Euclidean radius of $S$ is $1$, and that the Euclidean center of $S$ is $O = (0, \cdots, 0, \eta)$, where $\eta>-1$. Suppose $\vec{e} = \sum_{i=1}^n a_i \frac{\p}{\p y_i}$ is any vector with Euclidean norm 1. Then we may parametrize $\Sigma$ as $\{\vec{y} = O + \vec{e}: y_n>0\}$, which means the points on $\Sigma$ satisfy
$$
y_m = a_m, \qquad y_n = \eta + a_n,
$$
where $m$ runs from $1$ to $n-1$.

Denote the outward unit normal at $\vec{y}$ as $\nu = c\vec{e}$, then
$$1 = g(\nu, \nu) = y_n^{-2}\delta(c\vec{e}, c\vec{e}) = c^2y_n^{-2},$$
so 
$$\nu = y_n\vec{e} = (\eta + a_n)\sum_{i=1}^n a_i \frac{\p}{\p y_i}.$$
From Lemma \ref{hyperbolic mean curv}, the mean curvature with respect to this normal is $H = (n-1)\eta$. Then
$$
\begin{aligned}
&(x_0)_\nu-\frac{H}{n-1} x_0 = y_n\sum_{i=1}^n a_i \frac{\p}{\p y_i}x_0 - \eta x_0 = -\frac12\eta^2,\\
&(x_m)_\nu-\frac{H}{n-1} x_m = y_n\sum_{i=1}^n a_i \frac{\p}{\p y_i}x_m - \eta x_m = a_m - y_m = 0,\\
&(x_n)_\nu-\frac{H}{n-1} x_n = y_n\sum_{i=1}^n a_i \frac{\p}{\p y_i}x_n - \eta x_n = 1-\frac12\eta^2,
\end{aligned}
$$
where we make use of the relation $\sum_{i=1}^n a_i^2 = 1$. Notice that the boundary equation is linear, and we have
$$
\left((2-\eta^2)x_0 + \eta^2x_n\right)_\nu-\frac{H}{n-1} \left((2-\eta^2)x_0 + \eta^2x_n\right) = 0.
$$
So the region with boundary $\Sigma$ is a non-generic domain in $\mathbb{H}^n$, and $\operatorname{ker}\Phi^* = \operatorname{span}\{x_1, \cdots, x_{n-1}, (2-\eta^2)x_0 + \eta^2\, x_n\}.$
\end{proof}

The intersection with $\mathbb{H}^{n}$ of hyperplanes of $\mathbb{R}^{n}$ are also boundaries of non-generic domains in $\mathbb{H}^n$, as they can be obtained by inversion isometries from the spheres. In the next two propositions, we will determine the space of static potentials based on different positions of hyperplanes.

\begin{prop}
Consider the upper half-space model of $\mathbb{H}^n$. Let $\Sigma = P\cap\mathbb{H}^n$ be the intersection of $\mathbb{H}^n$ with a Euclidean $(n-1)$-plane $P\subset\mathbb{R}^n$ that is parallel to $\p\mathbb{H}^n$. Then
$$\operatorname{ker}\Phi^* = \operatorname{span}\{x_1, \cdots, x_{n-1}, x_0 - x_n\}.$$
Here $(x_0, x_1, \cdots, x_n)$ are the coordinates of $\mathbb{R}^{1, n}$ restricted to $\mathbb{H}^n$ in the hyperboloid model.
\end{prop}
\begin{proof}
We may parametrize $\Sigma$ as $\{y_n = c>0\}$, with the outward unit normal $\nu = y_n\frac{\p}{\p y_n}$. Recall the formula of the mean curvature under a conformal change $g = e^{2\phi}g_0$ is given by
\begin{equation}\label{conformal mean}
H=e^{-\phi}(H_0+(n-1)\langle\nabla \phi, \nu\rangle).
\end{equation}
So the mean curvature with respect to this normal is
$$
H = y_n\left(H_0 + (n-1)\left<\nabla(-\operatorname{ln}y_n), \frac{\p}{\p y_n}\right>_\delta\right) = -(n-1).
$$
Then for $m=1, \cdots, n-1$,
$$
\begin{aligned}
&(x_0)_\nu-\frac{H}{n-1} x_0 = y_n\frac{\p}{\p y_n}x_0 + x_0 = y_n >0,\\
&(x_m)_\nu-\frac{H}{n-1} x_m = y_n\frac{\p}{\p y_n}x_m + x_m = -\frac{y_m}{y_n} + \frac{y_m}{y_n} = 0, \\
&(x_n)_\nu-\frac{H}{n-1} x_n = y_n\frac{\p}{\p y_n}x_n + x_n = y_n >0.
\end{aligned}
$$
The result follows immediately.
\end{proof}

\begin{prop}
Consider the upper half-space model of $\mathbb{H}^n$. Let $\Sigma = P\cap\mathbb{H}^n$ be the intersection of $\mathbb{H}^n$ with a Euclidean $(n-1)$-plane $P\subset\mathbb{R}^n$ that makes an angle $\alpha\in(0, \pi)$ with $\p\mathbb{H}^n$. Up to isometry we may parametrize $\Sigma$ as $\{\operatorname{cos}\alpha\, y_n = \operatorname{sin}\alpha \, y_1: y_n > 0\}$, then
$$\operatorname{ker}\Phi^* = \operatorname{span}\{x_0, x_2, \cdots, x_n\}.$$
Here $(x_0, x_1, \cdots, x_n)$ are the coordinates of $\mathbb{R}^{1, n}$ restricted to $\mathbb{H}^n$ in the hyperboloid model.
\end{prop}
\begin{proof}
Up to isometry we may parametrize $\Sigma$ as $\{\operatorname{cos}\alpha\, y_n = \operatorname{sin}\alpha \, y_1: y_n > 0\}$, with the outward unit normal $\nu = y_n(-\operatorname{sin}\alpha\frac{\p}{\p y_1} + \operatorname{cos}\alpha\frac{\p}{\p y_n})$. So by (\ref{conformal mean}) the mean curvature with respect to this normal is
$$
H = y_n\left(H_0 + (n-1)\left<\nabla(-\operatorname{ln}y_n), -\operatorname{sin}\alpha\frac{\p}{\p y_1} + \operatorname{cos}\alpha\frac{\p}{\p y_n}\right>_\delta\right) = -(n-1)\operatorname{cos}\alpha.
$$
Then for $k=2, \cdots, n-1$,
$$
\begin{aligned}
&(x_0)_\nu-\frac{H}{n-1} x_0 = y_n(-\operatorname{sin}\alpha\frac{\p}{\p y_1} + \operatorname{cos}\alpha\frac{\p}{\p y_n})x_0 + \operatorname{cos}\alpha\, x_0 = 0,\\
&(x_1)_\nu-\frac{H}{n-1} x_1 = y_n(-\operatorname{sin}\alpha\frac{\p}{\p y_1} + \operatorname{cos}\alpha\frac{\p}{\p y_n})x_1 + \operatorname{cos}\alpha\, x_1 = -\operatorname{sin}\alpha<0,\\
&(x_k)_\nu-\frac{H}{n-1} x_k = y_n(-\operatorname{sin}\alpha\frac{\p}{\p y_1} + \operatorname{cos}\alpha\frac{\p}{\p y_n})x_k + \operatorname{cos}\alpha\, x_k = 0,\\
&(x_n)_\nu-\frac{H}{n-1} x_n = y_n(-\operatorname{sin}\alpha\frac{\p}{\p y_1} + \operatorname{cos}\alpha\frac{\p}{\p y_n})x_n + \operatorname{cos}\alpha\, x_n = 0.
\end{aligned}
$$
The result follows immediately.
\end{proof}


As a summary, we have the following classification result:
\begin{thm}[Classification of simple non-generic domains in $\mathbb{H}^n$]
There are four types of boundaries of non-generic domains in $\mathbb{H}^n$:\\
(i) horospheres, (ii) hyperspheres, (iii) geodesic spheres, (iv) intersections with $\mathbb{H}^n$ of hyperplanes of $\mathbb{R}^n$.
\end{thm}

We notice Hijazi-Montiel-Raulot \cite{H-M-R} had the following rigidity result for asymptotically hyperbolic manifolds with inner boundary:
\begin{prop}[Hijazi-Montiel-Raulot]
Let $(M^3,g)$ be a three-dimensional complete AH manifold with scalar curvature satisfying $R\ge -6$ and compact inner boundary $\Sigma$ homeomorphic to a $2$-sphere whose mean curvature is such that
$$
H \leq 2\sqrt{\frac{4 \pi}{\operatorname{Area}(\Sigma)}+1}.
$$
Then the energy-momentum vector $\mathbf{p}_g$ is time-like future-directed or zero. Moreover, if it vanishes then $(M^3,g)$ is isometric to the complement of a geodesic ball in $\mathbb{H}^3$.
\end{prop}

On the other hand, Almaraz-de Lima \cite{A-dL} had the following rigidity result for asymptotically hyperbolic manifolds with a non-compact boundary:

\begin{prop}[Almaraz-de Lima]
Let $(M, g, \Sigma)$ be an asymptotically hyperbolic spin manifold with boundary, with $R \ge -n(n - 1)$ and $H \ge 0$. Assume further that $g$ agrees with the reference hyperbolic metric $b$ in a neighborhood of infinity. Then $(M, g, \Sigma) = (\mathbb{H}^n_+,b,\p\mathbb{H}^n_+)$ isometrically.
\end{prop}

These results together give us a very good explanation of non-generic domains in $\mathbb{H}^{n}$.

\subsection{The Schwarzschild manifold}
In this subsection, we let $n\ge3$. Consider the Schwarzschild metric $g^{S}=\left(1+\frac{m}{2 r^{n-2}}\right)^{\frac{4}{n-2}} \delta$ on $\mathbb{R}^{n} \backslash\{\mathbf{0}\}$, where $r = |\vec{x}|$. When $m=0$, it becomes the Euclidean metric; when $m<0$, the metric is incomplete \cite[pp.~64-66]{Lee}. Therefore, we will just focus on the case when $m>0$.

We have already recalled that the umbilic hypersurfaces in $\mathbb{R}^n$ are spheres and hyperplanes, and that umbilic hypersurfaces remain umbilic after the conformal change \cite[p.~183 Exercise 6b]{dC}. This means the umbilic hypersurfaces in the Schwarzschild manifold are Euclidean spheres and hyperplanes.

\begin{lem}\label{sphere hyperplane mean}
Consider the Schwarzschild manifold $(\mathbb{R}^{n} \backslash\{\mathbf{0}\}, g^{S})$. Then the Euclidean spheres centered at $\mathbf{0}$ and the Euclidean hyperplanes containing $\mathbf{0}$ are the only possible boundaries of non-generic domains. In other words, the boundaries of non-generic domains in $(\mathbb{R}^{n} \backslash\{\mathbf{0}\}, g^{S})$ are symmetric with respect to $\mathbf{0}$.
\end{lem}
\begin{proof}
Let us denote $\alpha = \frac{m}{2r^{n-2}}$.

Suppose $\vec{e}$ is any vector with Euclidean norm 1. Then we may parametrize the Euclidean sphere centered at $\mathbf{0}$ as $\{\vec{x} = r\vec{e}\}$. Denote the outward unit normal at $\vec{x}$ as $\nu = c\vec{e}$, then
$$1 = g^S(\nu, \nu) = \left(1+\alpha\right)^{\frac{4}{n-2}}\delta(c\vec{e}, c\vec{e}) = c^2\left(1+\alpha\right)^{\frac{4}{n-2}},$$
so 
$$\nu = \left(1+\alpha\right)^{-\frac{2}{n-2}}\vec{e}.$$
Then by (\ref{conformal mean}) the mean curvature with respect to this normal is 
$$
\begin{aligned}
H & = \left(1+\alpha\right)^{-\frac{2}{n-2}}\left(H_0 + (n-1)\left<\nabla\left(\frac{2}{n-2}\operatorname{ln}(1+\alpha)\right), \vec{e}\right>_\delta\right)\\
	& = \left(1+\alpha\right)^{-\frac{2}{n-2}}\left(\frac{n-1}{r} - (n-1)\left<\frac{2\alpha\vec{x}}{r^2(1+\alpha)}, \vec{e}\right>_\delta\right)\\
	& = \left(1+\alpha\right)^{-\frac{2}{n-2}}\left(\frac{n-1}{r} - \frac{2(n-1)\alpha}{r(1+\alpha)}\right)\\
	& = (n-1)\frac{1-\alpha}{r\left(1+\alpha\right)^{1+\frac{2}{n-2}}},
\end{aligned}
$$
which is constant for fixed $r$.

On the other hand, the Euclidean spheres not centered at $\mathbf{0}$ are not boundaries of non-generic domains. This is mainly because the function $r$ is no longer a constant on the sphere, thus $H$ is not constant. More precisely, we may assume the sphere is centered at $\vec{a}=(0, \cdots, 0, 1)$ with radius $\rho$. Following the same calculation as above, we will see that the mean curvature depends on both the radius $\rho$ and the coordinate $x_n$ of the point, thus is not constant on the sphere.

Without loss of generality, we may parametrize a Euclidean hyperplane as $\{x_n = c\}$, then the outward unit normal is 
$$\nu = \left(1+\alpha\right)^{-\frac{2}{n-2}}\frac{\p}{\p x_n},$$
and by (\ref{conformal mean}) the mean curvature is 
$$
\begin{aligned}
H & = \left(1+\alpha\right)^{-\frac{2}{n-2}}\left(H_0 + (n-1)\left<\nabla\left(\frac{2}{n-2}\operatorname{ln}(1+\alpha)\right), \frac{\p}{\p x_n}\right>_\delta\right)\\
	& = -\left(1+\alpha\right)^{-\frac{2}{n-2}}\left( (n-1)\frac{2\alpha x_n}{r^2(1+\alpha)}\right)\\
	& = -(n-1)\frac{2\alpha}{r^2(1+\alpha)^{1+\frac{2}{n-2}}}c.
\end{aligned}
$$
Since $r$ is not a constant on the hyperplane, $\frac{2\alpha}{r^2(1+\alpha)^{1+\frac{2}{n-2}}}$ is not constant. Thus $H$ is not constant unless $c=0$. Finally, let us note the Euclidean hyperplanes containing $\mathbf{0}$ are complete and have two ends in $(\mathbb{R}^{n} \backslash\{\mathbf{0}\}, g^{S})$.
\end{proof}

As mentioned in \cite{C}, the function $u_0 = \frac{1-\frac{m}{2 r}}{1+\frac{m}{2 r}} \in \operatorname{ker} L_{g^S}^{*}$ is a possible static potential in $3$-dimensional Schwarzschild manifold. Analogously, the function $u = \frac{1-\frac{m}{2 r^{n-2}}}{1+\frac{m}{2 r^{n-2}}} \in \operatorname{ker} L_{g^S}^{*}$ is a possible static potential in $n$-dimensional Schwarzschild manifold. (This can also be found using the method introduced in \cite{B-M}.) Indeed, up to scaling, it is the only possible static potential in connected open domains (with our thanks to J. Corvino for bringing this to our attention). The main idea is to rewrite the static equation (\ref{static eqn}) in spherical coordinates and conclude that any possible static potentials must be radial. Now let us verify whether the function satisfies the boundary condition. 

\begin{prop}\label{Schwarzschild sphere}
Consider the Schwarzschild manifold $(\mathbb{R}^{n} \backslash\{\mathbf{0}\}, g^{S})$. Denote $\Sigma$ as the Euclidean sphere centered at $\mathbf{0}$ with radius $r_{\pm} = \left(\frac{(n-1) \pm \sqrt{n^2 - 2n}}{2}m\right)^{\frac{1}{n-2}}$. Then either region with boundary $\Sigma$ is a non-generic domain in $(\mathbb{R}^{n} \backslash\{\mathbf{0}\}, g^{S})$. Moreover, 
$$\operatorname{ker}\Phi^* = \operatorname{span}\left\{\frac{1-\frac{m}{2 r^{n-2}}}{1+\frac{m}{2 r^{n-2}}}\right\}.$$
\end{prop}
\begin{proof}
Denote $\alpha = \frac{m}{2r^{n-2}}$. Suppose $\vec{e}$ is any vector with Euclidean norm 1. Then we may parametrize the Euclidean sphere centered at $\mathbf{0}$ as $\{\vec{x} = r\vec{e}\}$. By Lemma \ref{sphere hyperplane mean}, the mean curvature is
\begin{equation}\label{mean curvature}
H = (n-1)\left(1+\alpha\right)^{-\frac{2}{n-2}}r^{-1}u.
\end{equation}
So
$$
\begin{aligned}
u_\nu-\frac{H}{n-1} u & = \left(1+\alpha\right)^{-\frac{2}{n-2}}\vec{e}\left(\frac{1-\alpha}{1+\alpha}\right) - \left(1+\alpha\right)^{-\frac{2}{n-2}}r^{-1}u^2\\
	& = \left(1+\alpha\right)^{-\frac{2}{n-2}}\left<\frac{2(n-2)\alpha\vec{x}}{r^2(1+\alpha)^2}, \vec{e}\right>_\delta - \frac{(1-\alpha)^2}{r(1+\alpha)^{2+\frac{2}{n-2}}}\\
	& = \frac{2(n-2)\alpha}{r(1+\alpha)^{2+\frac{2}{n-2}}} - \frac{\alpha^2 - 2\alpha + 1}{r(1+\alpha)^{2+\frac{2}{n-2}}}\\
	& = \frac{-\alpha^2 + 2(n-1)\alpha - 1}{r(1+\alpha)^{2+\frac{2}{n-2}}}.
\end{aligned}
$$
Thus $u_\nu-\frac{H}{n-1} u = 0$ if and only if
$$\alpha^2-2(n-1)\alpha+1 = 0,$$ 
that is,
$$\alpha = \frac{m}{2r^{n-2}} = (n-1) \pm \sqrt{n^2 - 2n}.$$
This means the region with boundary $\Sigma$ is a non-generic domain in $(\mathbb{R}^{n} \backslash\{\mathbf{0}\}, g^S)$.
\end{proof}

As a non-generic domain, one would expect it to have some special properties in the Schwarzschild manifold. It turns out that neither $r_{\pm}$ in Proposition \ref{Schwarzschild sphere} would give us the horizon of the Schwarzschild manifold, as the mean curvature $H$ cannot be $0$ (otherwise we would get $u \equiv 0$ from (\ref{mean curvature})). However, in the next proposition we will see that the two radii $r_{\pm}$ do have special geometric and physical meanings.

\begin{prop}
    The above two radii $r_{\pm}$ are the critical points of the mean curvature $H=H(r)$ on the Euclidean sphere (centered at $\mathbf{0}$) with radius $r$.
\end{prop}
\begin{proof}
As we can see from (\ref{mean curvature}), the mean curvature on the Euclidean sphere (centered at $\mathbf{0}$) with radius $r$ is a function depending only on $r$, that is, $H=H(r)$. 

Let us calculate its derivative $H'$, and denote $\alpha = \frac{m}{2r^{n-2}}$.
$$
\begin{aligned}
    H' & = mr^{1-n}(1+\alpha)^{-1}H - r^{-1}H + \frac{(n-2)mr^{1-n}}{(1-\alpha)(1+\alpha)}H\\
    & = H\frac{1}{r(1-\alpha)(1+\alpha)}\left[2\alpha(1-\alpha)-(1-\alpha)(1+\alpha)+(n-2)2\alpha\right]\\
    & = H\frac{1}{r(1-\alpha)(1+\alpha)}\left(-\alpha^2+2(n-1)\alpha-1\right).
\end{aligned}
$$
Note that $H\ne0$, so $H' = 0$ if and only if 
$$\alpha^2-2(n-1)\alpha+1 = 0.$$ 
This means, $H' = 0$ if and only if $r=r_{\pm}$.
\end{proof}

From (\ref{mean curvature}) we can see, when $r\rightarrow 0$ or $r\rightarrow \infty$, $H\rightarrow 0$. And $H=0$ on the horizon ($r^{n-2}=\frac{m}2$). Thus, 
\begin{itemize}
\item[(i)] $r_{-}$ (the one inside the horizon) gives us the minimal mean curvature $H_{\text{min}}$; 
\item[(ii)] $r_{+}$ (the one outside the horizon) gives us the maximal mean curvature $H_{\text{max}}$. 
\end{itemize}
Moreover, $(r_{-}\cdot r_{+})^{n-2} = \frac{m^2}{4}$; in other words, the two spheres are symmetric with respect to the horizon.

\begin{figure}
\begin{center}
\includegraphics[scale = 0.05]{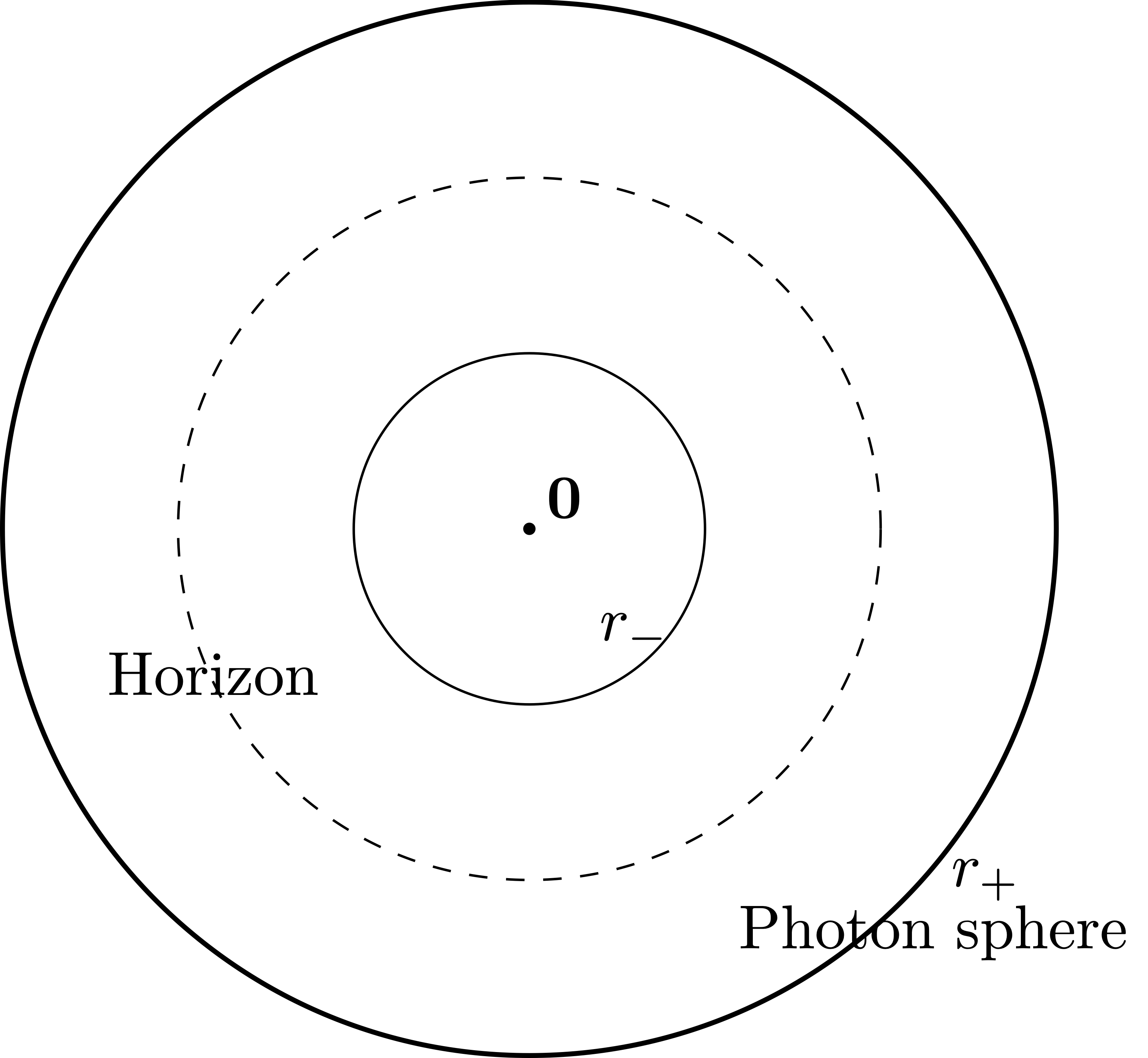}
\caption{The Schwarzschild photon sphere and the horizon}
\label{fig: photon sphere}
\end{center}
\end{figure}

In case $n= 3$, physicists call the one outside the horizon with radius $r_{+} = \frac{2+ \sqrt{3}}{2}m$ the Schwarzschild photon sphere \cite{Cederbaum, Claudel-V-E}, see Fig. \ref{fig: photon sphere}. Roughly speaking, it behaves like a borderline in the following sense:
\begin{itemize}
\item[(i)] Any endless geodesic starting at some point outside the photon sphere and initially directed outwards will continue outwards and escape to infinity. 
\item[(ii)] Any endless geodesic starting at some point between the horizon and the photon sphere and initially directed inwards will continue inwards and fall into the black hole. 
\item[(iii)] Any geodesic starting at some point on the photon sphere and initially tangent to the photon sphere will remain in the photon sphere. 
\end{itemize}

The Schwarzschild photon sphere thus models (an embedded submanifold ruled by) photons spiraling around the central black hole “at a fixed distance”. The Schwarzschild photon sphere and the notion of trapped null geodesics in general are crucially relevant for questions of dynamical stability in the context of the Einstein equations. Moreover, photon spheres are related to the existence of relativistic images in the context of gravitational lensing \cite{V-E}. On the other hand, Cederbaum and Galloway \cite{Cederbaum, Cederbaum-G} found rigidity results on the exterior region of the photon sphere, which happens to be a non-generic domain in our setting. Cruz and Nunes \cite{C-N} also had some rigidity results on the region between the horizon and the photon sphere.

\begin{prop}
Consider the Schwarzschild manifold $(\mathbb{R}^{n} \backslash\{\mathbf{0}\}, g^{S})$. Denote $\Sigma$ as a Euclidean hyperplane containing $\mathbf{0}$. Then either region with boundary $\Sigma$ is a non-generic domain in $(\mathbb{R}^{n} \backslash\{\mathbf{0}\}, g^{S})$. Moreover, 
$$\operatorname{ker}\Phi^* = \operatorname{span}\left\{\frac{1-\frac{m}{2 r^{n-2}}}{1+\frac{m}{2 r^{n-2}}}\right\}.$$
\end{prop}
\begin{proof}
Denote $\alpha = \frac{m}{2r^{n-2}}$. Without loss of generality, we may parametrize $\Sigma$ as $\{x_n = 0\}$, then by Lemma \ref{sphere hyperplane mean} the mean curvature on $\Sigma$ is $H=0$. So
$$
u_\nu-\frac{H}{n-1} u = \left(1+\alpha\right)^{-\frac{2}{n-2}}\frac{\p}{\p x_n}\left(\frac{1-\alpha}{1+\alpha}\right) + 0 = \left(1+\alpha\right)^{-\frac{2}{n-2}}\frac{2(n-2)\alpha}{r^2(1+\alpha)^2}x_n = 0.
$$
\end{proof}

As a summary, we have the following classification result:
\begin{thm}[Classification of simple non-generic domains in the Schwarzschild manifold]
There are two types of boundaries of non-generic domains in the Schwarzschild manifold $(\mathbb{R}^{n} \backslash\{\mathbf{0}\}, g^{S})$:\\
(i) the Euclidean spheres centered at $\mathbf{0}$ with radii $r_{\pm}$, (ii) the Euclidean hyperplanes containing $\mathbf{0}$.
\end{thm}

\subsection{Non-generic domains with multiple boundary components}\label{non-simple}
For the above examples in space forms and in the Schwarzschild manifold, we are just considering simple non-generic domains with only one boundary component. If a non-generic domain has multiple boundary components, then it can be viewed as the intersection of simple non-generic domains, as long as their boundaries do not intersect with each other. In this case, the space of static potentials is the intersection of those corresponding to each simple non-generic domain. But of course we will need to make sure the eventual space of static potentials has dimension at least $1$.

\begin{exam}
Consider the standard metric $g_{\mathbb{S}^n}$ on the unit sphere $\mathbb{S}^n$, and let $(x_1, \cdots, x_{n+1})$ denote the coordinates of $\mathbb{R}^{n+1}$ restricted to $\mathbb{S}^n$. For $\epsilon>0$ small, define $\Sigma_j\subset\mathbb{S}^n$ by, where $j=1,\cdots, n$,
$$
\Sigma_j = \{\mathbf{x}\in\mathbb{S}^n: x_{j} = 1-\epsilon\}.
$$
Then $\Sigma = \bigcup\limits_{j=1}^n \Sigma_j$ is a hypersurface with $n$ components. Denote the connected domain with boundary $\Sigma$ as $\Omega$. Then by Theorem \ref{sphere} we know $(\Omega, \Sigma)$ is a non-generic domain in $\mathbb{S}^n$, and moreover $\operatorname{ker}\Phi^* = \operatorname{span}\{x_{n+1}\}.$
\end{exam}

\section*{Acknowledgement}

I would like to thank my PhD advisor Rick Schoen for his wisdom and patience during the preparation of my dissertation. I would like to thank Justin Corvino for valuable discussions on Schwarzschild manifolds. I would like to thank the anonymous referees for carefully proofreading the article and for suggesting a number of changes which significantly contributed to the improvement of this final version.

\end{document}